\documentclass{article}
\usepackage[utf8]{inputenc}
\usepackage{amsmath}
\usepackage{amsthm}
\usepackage{amsfonts}
\usepackage[english]{babel}
\usepackage{mathtools}
\usepackage{hyperref}
\usepackage{comment}
\usepackage{authblk}
\usepackage{amscd}
\usepackage{amsfonts}
\usepackage{amssymb}
\usepackage{bbm}

\hypersetup{
    colorlinks=true,
    linkcolor=blue,
    filecolor=magenta,      
    urlcolor=cyan,
    }

\newcommand{\F}{\mathbb{F}}
\newcommand{\Q}{\mathbb{Q}}

\newcommand{\N}{\mathbb{N}}
\newcommand{\Z}{\mathbb{Z}}

\newcommand{\st}{{|}}

\newcommand{\id}{{\rm id}}

\newcommand{\frp}{\mathfrak{p}}
\newcommand{\frqqq}{\mathfrak{q}}
\newcommand{\Mod}[1]{\ (\mathrm{mod}\ #1)}

\usepackage{comment}

\usepackage{amsthm}

\theoremstyle{plain}
\newtheorem{theorem}{Theorem}[section]
\newtheorem{proposition}[theorem]{Proposition}
\newtheorem{lemma}[theorem]{Lemma}
\newtheorem{corollary}[theorem]{Corollary}

\theoremstyle{definition}
\newtheorem{definition}[theorem]{Definition}

\theoremstyle{remark}

\DeclarePairedDelimiter\floor{\lfloor}{\rfloor}

\newcounter{lthm}
\theoremstyle{plain}
\newtheorem{letterthm}[lthm]{Theorem}

\begin{document}

\title{On Landweber's unique factorization problem}
\author[1]{Adam Jones}
\author[2]{Elad Paran}

\affil[1]{University of Cambridge}
\affil[2]{The Open University of Israel}

%\author{Elad Paran\thanks{The Open University of Israel}} and Adam Jones \thanks{University of Cambridge}}
%Vietnam (vongocthieu@tdtu.edu.vn)}}
%\date\today

\maketitle

\begin{abstract}
\noindent Let $A$ be a regular unique factorization domain, and let $R = A[x_1,x_2,\ldots]$ be the ring of polynomials in countably many variables over $A$. We prove that the formal power series ring $R[[t]]$ is a unique factorization domain. This resolves a question raised by Landweber in 1974. The proof is based on a finite irreducibility theorem for Krull domains: if $R$ is a Krull domain and $f\in R[[t]]$ is irreducible, then $f$ is irreducible modulo a finite power of $t$.
\end{abstract}

\section{Introduction}\label{sec:introduction}

Let $R$ be a unique factorization domain (UFD). Is the formal power series ring $R[[t]]$ again a UFD? This question has a long history. In the complete local setting, where $R=K[[x_1,\ldots,x_n]]$ is a field or a formal power series ring over a field, Lasker gave a positive answer in 1905 \cite{Lasker05}. The general problem was studied extensively by Krull in the 1930s. Krull proved \cite{Krull38} that the answer is positive when $R$ is a discrete valuation ring with an infinite residue field, but suspected that the general answer is negative. Cohen settled the finite residue field case positively in 1946 \cite{Cohen46}. The first negative example was given by Samuel in 1961 \cite[\S4]{Samuel61}. Samuel's paper led to a flurry of activity, including work of Claborn, Danilov, Buchsbaum, and Salmon \cite{Claborn65,Danilov65,Buchsbaum61, Salmon64}, surrounding the question of when $R[[t]]$ is a UFD. For example, Samuel conjectured that $R[[t]]$ is a UFD whenever $R$ is a complete local domain, but this was disproved by Salmon \cite{Salmon64}. Samuel's study of obstructions to unique factorization also led to developments beyond power series rings: for instance, it led Grothendieck to prove, in SGA~2 \cite[Expos\'e XI, Corollaire 3.14]{Grothendieck68}, that if $R$ is a Noetherian local complete-intersection domain whose localizations at all prime ideals of height at most $3$ are UFDs, then $R$ itself is a UFD. This proved another conjecture that Samuel made in \cite[\S4]{Samuel61}; see also Lipman's account \cite[p.~531]{Lipman75}.

Through the accumulation of several works by different authors, including \cite{Samuel61,Buchsbaum61,Claborn65,Danilov65,Grothendieck68,Salmon64,HartshorneOgus74,Lipman75}, we now have a relatively good understanding of unique factorization in $R[[t]]$ in the Noetherian case. In particular, Samuel and Buchsbaum independently proved \cite[Theorem 2.1]{Samuel61}, \cite[Theorem 3.2]{Buchsbaum61} that if $R$ is a regular UFD\footnote{Recall that a Noetherian UFD is called {\it regular} if all its local rings are regular local rings.}, then $R[[t]]$ is again a UFD. This includes the case of polynomial rings in finitely many variables over a field, and more generally polynomial rings in finitely many variables over any regular UFD.

The non-Noetherian case is much less understood. In particular, in \cite{Landweber74} Landweber raised the following question: Let $R$ be a regular UFD. Is the ring $R[x_1,x_2,\ldots][[t]]$ a UFD? Landweber solves a ``graded'' version of the problem for a certain subring of $R[x_1,x_2,\ldots][[t]]$, and notes the difficulty of the problem for the full ring; see \cite[Remark 3.9]{Landweber74}. Landweber notes that the answer is unknown even in the case where $R = K$ is a field. This problem was later highlighted in surveys by Anderson in 1995 \cite{Anderson95} and by Gilmer in 2006 \cite{Gilmer06Questions}. In the years since, there have been numerous works studying factorization and irreducibility in power series rings, for example \cite{PS06,BGW07,BG08,KangOh09,BGW12,BGW14,BGW14b,Elliott14,DS21,GTV25}, some of which mention Landweber's problem (e.g. in the introduction of \cite{Elliott14}), yet the problem remained unsolved.

In this paper we solve Landweber's problem in full generality. We prove:

\begin{letterthm}[Theorem \ref{regular_UFD_polynomial_theorem}]\label{main_theorem}
Let $A$ be a regular UFD, let $I$ be any set, and let $R=A[x_i\mid i\in I]$ be the ring of polynomials in $|I|$ independent variables over $A$. Then $R[[t]]$ is a UFD.
\end{letterthm}

Given a domain $R$, a non-unit power series $f\in R[[t]]$, and an integer $n\geq 1$, we say that $f$ is \emph{irreducible modulo $t^n$} if there do not exist non-units $g,h\in R[[t]]$ such that $f\equiv gh \Mod{t^n}$. The main ingredient in the proof of Theorem \ref{main_theorem} is the following finite irreducibility theorem:

\begin{letterthm}[Theorem \ref{finite_irreducibility_section}]\label{finite_irreducibilty_theorem}
Let $R$ be a Krull domain. Let $f$ be an irreducible power series in $R[[t]]$. There exists $n\geq 1$ such that $f$ is irreducible modulo $t^n$.
\end{letterthm}

In the special case where $R$ is a characteristic-zero UFD in which every integer is a unit, Theorem \ref{finite_irreducibilty_theorem} was proved by Bayart in \cite[\S IV, Th\'eor\`eme (T)]{Bayart81}; here we prove the theorem in the general context of Krull domains. 

Recall that a commutative domain is a UFD if and only if it is atomic and every irreducible element is prime \cite[Chapter VII, \S1]{BourbakiCA}. If $R$ is a Krull domain, then $R[[t]]$ is again a Krull domain \cite[Chapter VII, \S1]{BourbakiCA}, and in particular is atomic. Thus, after proving Theorem \ref{finite_irreducibilty_theorem}, it remains to prove primeness for irreducibles of finite height. We do this in the following general form:

\begin{letterthm}[Theorem \ref{finite_stage_retract_criterion}]\label{finite_stage_retract_intro}
Let $R$ be a Krull domain. Suppose that for every finite subset $E\subset R$ there exists a subring $S\subseteq R$ and a retraction $\rho:R\to S$ such that $E\subseteq S$ and $S[[t]]$ is a UFD. Then $R[[t]]$ is a UFD.
\end{letterthm}

Theorem \ref{main_theorem} follows from Theorem \ref{finite_stage_retract_intro}: for a finite set of elements of $A[x_i\mid i\in I]$, only finitely many variables occur; the corresponding finite polynomial subring over $A$ is a regular UFD, hence its power series ring is a UFD by the theorem of Samuel--Buchsbaum, and the map sending all other variables to $0$ is the required retraction.

The bulk of the paper is devoted to the proof of Theorem \ref{finite_irreducibilty_theorem}. We note that in the special case $R=\mathbb Z$, one can give a short proof; see Proposition \ref{finite_quotient}. In the special case where $R$ is a Noetherian, Henselian, excellent local domain, the theorem also follows from Artin's strong approximation theorem \cite{Artin69,PfisterPopescu1975}.

The proof of Theorem \ref{finite_irreducibilty_theorem} proceeds by contraposition. We assume that $f$ has factorizations modulo arbitrarily high powers of $t$, and then prove that sufficiently compatible partial factorizations yield a proper factorization of $f$. This is done via a K{\"o}nig's lemma based argument, see Lemma \ref{tree}. The main technical input needed to obtain compatibility is a ``primality of finite height'' property for power series. This property holds over discrete valuation rings, and is proved in \S\ref{sec:quasi_primality}. It does not hold over arbitrary Krull domains or even arbitrary UFDs, as explained in the remark at the end of \S\ref{sec:quasi_primality}. We therefore first prove Theorem \ref{finite_irreducibilty_theorem} over DVRs in \S\ref{sec:dvr_case}. In \S\ref{sec:krull_proof}, we then use a ``local-global" argument and the Krull intersection property to extend the result to all Krull domains. Finally, \S\ref{sec:finite_means_prime_proof} proves the finite-stage retraction criterion mentioned above and deduces Theorem \ref{main_theorem}.

\textbf{Acknowledgements:} The first author is grateful to the Heilbronn Institute for Mathematical Research for supporting this research. The authors thank Thieu N. Vo and Arno Fehm for helpful discussions on the topic of this paper.

\section{Preliminaries}\label{sec:preliminaries}

We begin by fixing standard terminology conventions: 
Throughout this work, by a {\it domain} we mean a commutative integral domain. If $R$ is a domain, a non-unit non-zero element $p \in R$ is called {\it prime} if $p|ab$ implies $p|a$ or $p|b$ for all $a,b \in R$. A non-zero, non-unit element $a \in R$ is called {\it irreducible} if $a$ cannot be written as a product of two non-units in $R$. Whenever we discuss a {\it unique factorization domain}, we shall write UFD in short, and similarly we shall write DVR for a {\it discrete valuation ring}. 

Given a domain $R$, power series $a,b \in R[[t]]$ and an integer $n \geq 0$, we shall write $a|_nb$, or $a| b \Mod{t^n}$, to mean that $a$ divides $b$ modulo $t^n$, that is, $b \in (a,t^n)$ (note that $a|_0 b$ always holds). We shall write $a \equiv _n b$ to mean that $a \equiv b\Mod{t^n}$. We shall write $a \cong_n b$ to mean that $(a,t^n) = (b,t^n)$; equivalently, $a|_n b$ and $b|_na$; equivalently, there exists $u \in R[[t]]^\times$ such that $a \equiv_n bu$.

Throughout the text, we denote the set of positive integers by $\N$.

\begin{lemma}\label{lem_a|b} Let $R$ be a domain and let $a,b \in R[[t]]$. Then $a|_nb$ for all $n \in \N$ if and only if $a|b$ in $R[[t]]$. \end{lemma}
\begin{proof}
    If $a = 0$ the claim holds trivially. Suppose that $a \neq 0$ and that for all $n \geq 1$ there exists $q^{(n)} \in R[[t]]$ such that $aq^{(n)} \equiv_n b $. Then $t^n | a(q^{(n+1)}-q^{(n)})$ for all $n$, which implies that  $t^{n-m} | q^{(n+1)}-q^{(n)}$, for all $n\geq m$, where $m \geq 0$ is the $t$-adic value of $a$. Thus the sequence $\{q^{(n)}\}_{n \geq 1}$ converges $t$-adically to a power series $q \in R[[t]]$, satisfying $aq = b$.
\end{proof}

%If a power series has a non-prime constant term, then it may or may not be irreducible. This leads us to define the notion of a {\it height} of an irreducible power series:

Recall that in general, an integral domain is a UFD if and only if it is atomic and every irreducible in the ring is prime. If $R$ is a UFD then in particular it is a Krull domain, hence so is $R[[t]]$, and in particular it is atomic (see for example the introduction of \cite{Coy22}). 

Thus to prove or disprove that for a UFD $R$ the ring $R[[t]]$ is a UFD, one must understand whether all irreducible elements in $R[[t]]$ are prime. There is at least one type of irreducible power series that is always prime -- those whose constant term is a prime (clearly, such power series are irreducible in $R[[t]]$):

\begin{proposition}
\label{prime2} Let $f = f_0+f_1t+\ldots \in R[[t]]$ where $R$ is an arbitrary domain. If $f_0$ is prime in $R$ then $f$ is prime in $R[[t]]$. \end{proposition}
\begin{proof}
Let $C = R[[t]]/(f)$. It is immediate to verify that multiplication by $t$ is injective on $C$. Note also that $C$ is $t$-adically separated. That is, $\bigcap_{n \geq 1}t^n C = \{0\}$. Indeed, if  $\bar{a} \in \bigcap_{n \geq 1}t^n C$, $a \in (f,t^n)$ for all $n$, hence $f|a$ in $R[[t]]$ by Lemma \ref{lem_a|b}, hence $\bar{a} = 0$.

Now, we have 
$$C/tC  = R[[t]]/(f,t) = R/(f_0),$$
hence $C/tC$ is a domain. We prove that $C$ is a domain. Let $a \mapsto \bar{a}$ denote the reduction map from $R[[t]]$ to $C = R[[t]]/(f)$. Let $\bar{a},\bar{b}$ be non-zero element of $C$, for some $a,b \in R[[t]]$. Then since $C$ is $t$-adically separated, there exist maximal integers $r,s \geq 0$ such that $\bar{a} \in t^r C, \bar{b} \in t^s C$, respectively. Let $c,d \in R[[t]]$ with $\bar{c},\bar{d} \notin tC$ such that $\bar{a} = t^r\bar{c},\bar{b} = t^s \bar{d}$. Then the images of $\bar{c},\bar{d}$ in $C/tC$ are non-zero. Since $C/tC$ is a domain, we have $\bar{c}\bar{d} \notin tC$, and in particular $\bar{c}\bar{d} \neq 0$ in $C$. Since multiplication by $t$ is injective on $C$, so is multiplication by $t^{r+s}$, hence $\bar{a}\bar{b} = t^{r+s}\bar{c}\bar{d} \neq 0$. Thus $C$ is a domain, hence $f$ is prime in $R[[t]]$.

\end{proof}

(We note that the above simple lemma appears in the literature in special cases, for example when $R$ is a PID \cite[Section 3]{Elliott14}. We could not find a reference for the lemma in the general form above.)

\begin{definition}
       
    Let $0\neq f \in R[[t]]$ be a non-unit. We say that $f$ is {\it irreducible of finite height} if $f$ is irreducible modulo $t^k$ for some $k \geq 1$. If such an integer exists, we call the minimal such $k$ the {\it irreducibility height} of $f$, denote it ht$(f)$, and also say that $f$ is irreducible of height $k$. 
\end{definition}

\noindent Clearly, if $f \in R[[t]]$ is irreducible of finite height, then $f$ is irreducible. The converse is not generally true; see the example following Corollary \ref{cor:finite_height_ufd_criterion}. %Let us pose the following problem: characterize the integral domains $R$ for which every irreducible element in $R[[t]]$ is irreducible of finite height.

Let us note a trivial finite factorization result:

\begin{lemma}\label{zero_coefficient}
    Let $R$ be a domain with $f_0 = 0$. If $f$ is irreducible, then $f$ is irreducible of height 2. 
\end{lemma}
\begin{proof}
If $f_1 = 0$ then $t^2|f$ so $f$ is reducible. Thus $f_1 \neq 0$ and we have $f = t\cdot (t^{-1}f)$, hence the constant term $f_1$ of $t^{-1}f$ must be a unit in $R$. If $a,b \in R[[t]]$ satisfy $f \equiv_2 ab$, then without loss of generality $a_0= 0$, and $f_1 = a_1b_0$, hence both $a_1,b_0$ are units in $R$. Thus $b$ is a unit in $R[[t]]$, and we deduce that $f$ is irreducible mod $t^2$.  
\end{proof}

\section{A factorization lemma}\label{sec:factorization_lemma}

Let us formulate the following factorization lemma, whose proof is a consequence of K{\"o}nig's lemma. Later, the proof of our main result is achieved by establishing the conditions of this lemma, first in the DVR case and later in the general Krull domain case. 

\begin{lemma}\label{tree}
    Let $R$ be a domain, let $f \in R[[t]]$ and let $d,e$ be fixed non-zero elements in $R$. For each $n \in \N$ let $C_n$ be the set of all pairs $(a,b) \in R[[t]]^2$ satisfying $f \equiv_n ab, a_0 = d, b_0 = e$. For any $n,\ell \in \N$, let 
    $$I_{n,\ell} = \{(a,t^\ell)|(a,b) \in C_n\}.$$
    Suppose that: \begin{enumerate}
        \item The sets $C_n$ are non-empty for arbitrarily large $n$.
        \item For every $\ell \in \N$ there exists $M_\ell \in \N$ such that $\bigcup_{n \geq M_\ell} I_{n,\ell}$ is finite.
    \end{enumerate}
    Then there exist $a,b \in R[[t]]$ with $a_0 = d, b_0 = e$ such that $f  = ab$. In particular, if $d,e \notin R^\times$ then $f$ is reducible in $R[[t]]$. 
\end{lemma}
\begin{proof}
    Since every $C_n$ is non-empty for some $n$, we have, in particular, $f_0 = de$. For every $\ell \in \N$ let ${\cal{F}}_\ell = \bigcup_{n \geq M_\ell} I_{n,\ell}$, and let us define a set of ideals $T_\ell$ as follows: An ideal $I \subset R[[t]]$ belongs to $T_\ell$ if for arbitrarily large $n \in \N$ there exist $(a,b) \in C_n$ such that $I = (a,t^\ell)$. 
    
    Note that $T_\ell \subseteq {\cal{F}}_\ell$. Indeed, if $I \in T_\ell$ then $I = (a,t^\ell)$ for some $n \geq M_\ell$ and $(a,b) \in C_n$, hence $I \in {\cal{F}}_\ell$. In particular, $T_\ell$ is finite. 

    Note that $T_\ell$ is non-empty for all $\ell$. Indeed, by assumption the sets $C_n$ are non-empty for arbitrarily large $n$, while only finitely many ideals occur in ${\cal{F}}_\ell = \bigcup_{n \geq M_\ell} I_{n,\ell}$, hence at least one of these ideals must occur for arbitrarily large $n$. 

    Now let $T = \bigcup_{\ell \geq 1}T_\ell$. Note that $T_1 =\{ (d,t)\}$ is a singleton. We now view $T$ as a directed tree, as follows: for an ideal $I \in T_{\ell+1}$ we define its parent as the ideal $I+(t^\ell)$. Clearly, $I+(t^\ell) \in T_\ell$. By K{\"o}nig's lemma \cite{Konig27}, there exists an infinite branch in $T$. That is, a sequence $I_1,I_2,\ldots$ such that $I_{\ell+1}+(t^\ell) = I_\ell$ for all $\ell \geq 1$. For each $\ell$ choose any factor $a^{(\ell)} \in R[[t]]$ of $f$ modulo $t^\ell$, with $a^{(\ell)}_0 = d$ such that $I_\ell = (a^{(\ell)},t^\ell)$. Then $$(a^{(\ell)},t^\ell) = (a^{(\ell+1)},t^{\ell+1})+(t^\ell) =(a^{(\ell+1)},t^{\ell}).$$  Thus $a^{(\ell+1)} \cong_\ell a^{(\ell)}$ for all $\ell$. For each $\ell \in \N$, let $b^{(\ell)} \in R[[t]]$ such that $a^{(\ell)}b^{(\ell)} \equiv_\ell f$. Then necessarily $b^{(\ell)}_0 = e$. By multiplying with suitable units in $R[[t]]$ at every level, we may assume that $a^{(\ell+1)} \equiv_\ell a^{(\ell)}$ for all $\ell$. Thus $(a^{(\ell)})_{\ell \geq 1}$ is a $t$-adically Cauchy sequence. From
    $$a^{(\ell)}b^{(\ell)} \equiv_\ell f \equiv_\ell a^{(\ell+1)}b^{(\ell+1)} \equiv_\ell a^{(\ell)}b^{(\ell+1)}$$ we get that $$a^{(\ell)}(b^{(\ell)}-b^{(\ell+1)}) \equiv_\ell 0,$$
and since $a^{(\ell)}_0= d \neq 0$, it follows that $b^{(\ell)}-b^{(\ell+1)} \equiv_\ell 0$. Thus the sequence $(b^{(\ell)})_{\ell \geq 1}$ is also Cauchy.  The respective limits $a,b \in R[[t]]$ then satisfy $a_0 = d, b_0 = e$ and $f = ab$.\end{proof}

Let us now demonstrate the utility of Lemma \ref{tree} in a special case. 

\begin{proposition}\label{finite_quotient}
    Let $R$ be a domain. Suppose that for every non-zero $r \in R$ with $r \notin R^\times$, the quotient ring $R/(r)$ is finite. Then every irreducible element in $R[[t]]$ is irreducible of finite height. 
\end{proposition}
\begin{proof}
Let $\ell \in \N$ and let $0 \neq d \in R$. We claim that there exist finitely many ideals in $R[[t]]$ of the form $(a,t^\ell)$ with $a_0 = d$.  Indeed, if $d \in R^\times$, then for every $a \in R[[t]]$ with $a_0 = d$ we have $(a,t^\ell) = R$, and the claim holds trivially. Suppose that $d \notin R^\times$, and write $a = d+tb$, $b \in R[[t]]$. Then $d^\ell = (a-tb)^\ell$, and by the binomial expansion we see that $d^\ell \in (a,t^\ell)$. Thus every ideal of the form $(a,t^\ell)$ contains the fixed ideal $(d^\ell,t^\ell)$, and therefore such ideals correspond to ideals in the quotient ring $R[[t]]/(d^\ell,t^\ell)$.  The latter ring is finite, since by our assumption the quotient $R/(d^\ell)$ is finite, and in particular $R[[t]]/(d^\ell,t^\ell)$ contains finitely many ideals. This proves our claim. 

Now, let $f \in R[[t]]$ be irreducible. If $f_0 = 0$, then $f$ is irreducible of finite height by Lemma \ref{zero_coefficient}. Suppose that $f_0 \neq 0$. We show that $f$ is irreducible of finite height. Suppose to the contrary that modulo every $n \geq 1$, there exists $a^{(n)},b^{(n)}$ with non-unit constant terms, such that $f \equiv_n a^{(n)}b^{(n)}$. In each such factorization, we have $a^{(n)}_0b^{(n)}_0 = f_0$. Note that $f_0$ has finitely many factorizations up to units, since these correspond to ideals containing $(f_0)$. Thus
we may replace $(a^{(n)},b^{(n)})_{n \geq 1}$ with a suitable subsequence and multiply with suitable units in $R[[t]]$, to assume that $d = a^{(n)}_0$ and $e = b^{(n)}_0$ are independent of $n$, where $d,e$ are non-units in $R$. Let $\ell \in \N$, and for any $n \in \N$, let $I_n$ be the set of ideals of the form $(a,t^\ell)$, where $a$ is a factor of $f$ mod $t^n$ with $a_0 = d$. Then $I_n$ is contained in the set of all ideals of the form $(a,t^\ell)$ with $a_0 = d$, which is finite, as noted above. Hence $$\bigcup_{n \geq 1}I_n$$
is finite. This establishes the conditions of Lemma \ref{tree}, and hence there exists a factorization $f = ab$ with $a_0 = d,b_0= e$, a contradiction.

\end{proof}

Proposition \ref{finite_quotient} establishes Theorem \ref{finite_irreducibilty_theorem} in the special cases where $R = \Z$ is the ring of integers, or $R = \F_q[x]$ is the ring of polynomials in one variable over a finite field $\F_q$. We note that there are several papers focused on factorization and reducibility in $\mathbb Z[[t]]$, for example \cite{BGW07,BG08,BGW12,BGW14,BGW14b,Elliott14,DS21}, but Theorem \ref{finite_irreducibilty_theorem} for $R = \Z$ does not seem to appear in any of them.

Let us also note that in the special case where $R$ is a Noetherian, excellent Henselian local domain, so is the ring $R[[t]]$, and in this case Theorem \ref{finite_irreducibilty_theorem} follows as an immediate consequence of Artin's strong approximation theorem (see \cite[Theorem 3.7]{ScheidererWeightedSOS}, or \cite[\S1, Definition 1.1, and the main theorem on pp.~145--146]{PfisterPopescu1975}).

\section{$C$-prime elements}\label{sec:quasi_primality}

In this section we prove that if a power series over a DVR is irreducible of finite height, then it is also ``prime of finite height'', in the sense of $C$-prime elements that we will make precise below.

\begin{definition} 
%\begin{itemize}
%\item We say that $A$ is \emph{irreducible of finite height} if there exists $d\in\mathbb{N}$ such that $A$ is irreducible modulo $t^d$. We call the minimal such $d$ the \emph{height} of $A$ (denoted ht$(A)$).
%\item 
Let $R$ be a domain and let $0\neq a\in R[[t]]$ be a non-unit. Let $C\in\mathbb{N}$. We say that $a$ is \emph{$C$-prime} if whenever $a|_n b_1b_2$ in $R[[t]]$, there exist $i,j\leq n$ such that $a|_ib_1$, $a|_jb_2$ and $i+j> n-C$.

%\end{itemize}
\end{definition}

\noindent\textbf{Note:} If $a\in R[[t]]$ is $C$-prime, then $a$ is $D$-prime for any integer $D\geq C$.

\begin{lemma}\label{prime bound}
Suppose $a\in R[[t]]$ is $C$-prime, and for $n\in\mathbb{N}$ set $n(C):=\max\{0,\lceil\frac{n-C}{2}\rceil\}$. Then if $a|_nb_1b_2$ then $a|_{n(C)}b_1$ or $a|_{n(C)}b_2$.
\end{lemma}

\begin{proof}

We may assume that $n>C$, so suppose to the contrary that if $i,j\geq 0$ with $a|_ib_1$ and $a|_jb_2$ then $i,j<\frac{n-C}{2}$. Then $i+j<n-C$. But $a$ is $C$-prime, so we know that we can choose $i,j$ such that $i+j>n-C$ -- contradiction.\end{proof}

The terminology ``$C$-prime" is justified by the following lemma:

\begin{lemma}\label{C-prime means prime} Let $R$ be a domain and let $0 \neq a \in R[[t]]$ be $C$-prime for some $C \geq 1$. Then $a$ is prime in $R[[t]]$.
    
\end{lemma}
\begin{proof}
    Let $b_1,b_2 \in R[[t]]$ such that $a|b_1b_2$. Then $a|_nb_1b_2$ for all $n \geq 1$. Thus there exist $i,j$ such that $a|_ib_1$, $a|_jb_2$ and $i+j> n-C$. Since this holds for unbounded $n$, there exists arbitrarily large $i$ such that $a|_i b_1$, or arbitrarily large $j$ such that $a|_j b_2$. In the first case it follows that $a|b_1$, and in the second case that $a|b_2$, by Lemma \ref{lem_a|b}.
\end{proof}

\begin{lemma}\label{prime implies irreducible (finite height)}
If $f\in R[[t]]$ is $C$-prime for some $C\geq 1$, then $f$ is irreducible of finite height, and \emph{ht}$(f)\leq C+1$.
\end{lemma}

\begin{proof}

Since $f$ is $C$-prime, we know that it is prime by Lemma \ref{C-prime means prime}. Since all primes in a domain are irreducible, this implies that $f$ is irreducible. In particular, if $f_0= 0$, then $f$ is irreducible of height 2 by Lemma \ref{zero_coefficient}, so we may assume that $f_0\neq 0$.\\

\noindent Suppose for contradiction that $f$ is not irreducible modulo $t^{C+1}$. Choose non-units $a,b\in R[[t]]$ such that $f\equiv ab\Mod{t^{C+1}}$.

Then $f|_{C+1}ab$, so applying $C$-primality we see there exist $i,j\leq C+1$ such that $f|_ia$, $f|_jb$ and $i+j>C+1-C>0$. Thus $i>0$ or $j>0$. Assuming without loss of generality that $i>0$, it follows that $a\equiv fa'\Mod{t^i}$, so $f\equiv ab\equiv fa'b\Mod{t^i}$. Since $f_0\neq 0$, it follows that $a'b\equiv 1\Mod{t^i}$, and hence $b$ is a unit in $R[[t]]$, a contradiction.\end{proof}

%\noindent Let us remark that a consequence of Theorem \ref{finite_irreducibilty_theorem} is that for any Krull domain $R$, to prove that $R[[t]]$ is a UFD, it suffices to show that if $f\in R[[t]]$ is irreducible of finite height, then it is $C$-prime for some $C\geq 1$. The remainder of this section will be dedicated to proving this statement over a discrete valuation ring.

For the rest of this section, let $V$ be a DVR, and denote its corresponding valuation by $v$. We shall use capital letters to denote power series in $V[[t]]$, as opposed to lowercase letters with which we shall denote power series over a general domain elsewhere, in order to add further visual distinction between the two situations.

The purpose of this section is to prove a quantitative form of primality over a DVR: if $A$ is irreducible modulo a fixed power $t^d$ and if $v(A_0)$ is fixed, then $A$ is $C$-prime for a constant $C$ which depends only on $d, v(A_0)$. The existence of such a constant will be used in the next section.

\begin{lemma}\label{auxillary}
If $A,B_1,B_2\in V[[t]]$ such that $v(A_0)=m$ and $A|_n B_1B_2$, and let $i,j\geq 0$ be maximal such that $i,j\leq n$, $A|_i B_1$ and $A|_j B_2$, and assume that $i+j<n$. Then writing $B_1=AX+t^iS$ and $B_2=AY+t^jT$, it follows that $0<v(S_0),v(T_0)<m$ and there exists $W\in V[[t]]$ such that $v(W_0)<m$ and $$AW\equiv ST\Mod{t^{n-i-j}}.$$
\end{lemma}

\begin{proof}

Writing $B_1B_2\equiv HA\Mod{t^n}$, we see that $AB_2X+t^iSB_2\equiv HA\Mod{t^n}$. So since $i<n$ we have $A(H-B_2X)\in (t^i)$, and hence $H-B_2X\in (t^i)$.

Setting $Z:=\frac{H-B_2X}{t^i}$ we have $$AZ\equiv SB_2\equiv S(AY+t^jT)\Mod{t^{n-i}},$$ so since $j<n-i$ it follows that $Z-SY\in (t^j)$. Setting $W:=\frac{Z-SY}{t^j}$, we have $$AW\equiv TS\Mod{t^{n-i-j}}$$ Moreover, since $v(A_0)=m$, if $v(S_0)\geq m$, then $A\mid_{i+1} t^iS$, so $t^i S\in (A,t^{i+1})$, and thus $B_1=AX+t^iS\in (A,t^{i+1})$, contradicting the maximality of $i$. So $v(S_0)<m$, and similarly $v(T_0)<m$. So if $k:=v(W_0)$, then since $A_0W_0=T_0S_0$, we see that $k=v(S_0)+v(T_0)-m<2m-m=m$. 

Also, since $v(S_0)<m$, $v(T_0)=m-v(S_0)+v(W_0)>m-m+v(W_0)\geq 0$, and similarly $v(S_0)>0$.\end{proof}

\begin{corollary}\label{valuation_one_case}
If $A\in V[[t]]$ with $v(A_0)=1$, then $A$ is 1-prime.
\end{corollary}

\begin{proof}

Suppose $B_1,B_2\in V[[t]]$ satisfy $A|_n B_1B_2$. Then if $i,j\geq 0$ are maximal such that $i,j\leq n$, $A|_iB_1$ and $A|_jB_2$,  then suppose for contradiction that $i+j<n$. Using Lemma \ref{auxillary}, we can write $B_1=AX+t^iS$ and $B_2=AY+t^jT$, with $0<v(S_0),v(T_0)<v(A_0)=1$, a contradiction.
Hence $i+j\geq n>n-1$, so $A$ is 1-prime.\end{proof}

\noindent Now, if $A\in V[[t]]$ and $v(A_0)=1$, then $A$ is irreducible of height $1$, and it follows from Corollary \ref{valuation_one_case} that $A$ is $1$-prime.

\begin{definition}
Given $m>1$, we say that $L\in\mathbb{N}$ \emph{controls primality to degree $m$} if for all $A\in V[[t]]$, irreducible of finite height with $v(A_0)<m$, $A$ is $L\text{ht}(A)$-prime.
\end{definition}

\noindent For example, using Corollary \ref{valuation_one_case}, we see that 1 controls primality to degree 2.

%and let us suppose that we have defined integers $$1=N_1<N_2<\dots <N_{m-1}$$ such that if $A$ is irreducible of height $d$ and $v(A_0)=k$ for $k<m$, then $A$ is $N_kd$-prime.

\begin{lemma}\label{extra}
If $d,d',m,s,N,L\in\mathbb{N}$, $m>1$, $s<m$, $A,W_1,\dots,W_s,S,T\in V[[t]]$ non-units, with $A$ irreducible of height $d$, $W_1,\dots,W_s$ irreducible of height at most $d'$, $v(A_0)=m$, $v(T_0),v(S_0)<m$ and $AW_1\cdots W_s\equiv TS\Mod{t^N}$. If $L$ controls primality to degree $m$, then $$N\leq 2^md+(2^m-1)Ld'$$
\end{lemma}

\begin{proof}

If $N \leq d$ the bound holds trivially, so we shall assume that $N > d$. 

Since the constant term of $W_1\cdots W_s$ has value $v(S_0)+v(T_0)-m<2m-m=m$, we may choose $k_i<m$ such that $v(W_{i,0})=k_i$ for each $i$. Also, fix $d_i\in\mathbb{N}$ such that $W_i$ is a non-unit, irreducible of height $d_i$, and we know that $d_i\leq d'$ for all $i$ by assumption.\\   

\noindent Define $M_0:=\max\{N,d\}$, and for each $i\geq 0$, define $M_{i+1}=\max\{\lceil\frac{M_{i}-d_{i+1}L}{2}\rceil,d\}$. We will first prove that $M_\ell=d$ for some $\ell\leq s$ (and we will fix $\ell$ minimal such that this is satisfied). 

Suppose that $M_i>d$ for all $i<s$, then we will prove that $M_s=d$. Let us suppose for contradiction that $M_s>d$.\\

\noindent Assume that for some $i\geq 0$, $AW_{i+1}\dots W_{s}\equiv T^{(i)}S^{(i)}\Mod{t^{M_i}}$, where $T^{(i)}|_{M_i} T$ and $S^{(i)}\mid_{M_i} S$ (which we know to be true when $i=0$ since $M_0>d$). Thus $W_{i+1}|_{M_i}T^{(i)}S^{(i)}$.

But $W_{i+1}$ is irreducible of height $d_{i+1}$ and $v(W_{i+1,0})=k_{i+1}<m$, so $W_{i+1}$ is $Ld_{i+1}$-prime by assumption.\\ 

\noindent Since $M_{i+1}>d$, we have $M_{i+1}=\lceil\frac{M_{i}-d_{i+1}L}{2}\rceil$ thus by Lemma \ref{prime bound} $$W_{i+1}|_{M_{i+1}} T^{(i)}\text{ or }W_{i+1}|_{M_{i+1}} S^{(i)}$$ Without loss of generality, we will assume that $W_{i+1}|_{M_{i+1}} T^{(i)}$.\\ 

\noindent Writing $T^{(i)}\equiv W_{i+1}T^{(i+1)}\Mod{t^{M_{i+1}}}$ and $S^{(i+1)}:=S^{(i)}$, we see that $$AW_{i+1}W_{i+2}\cdots W_s\equiv W_{i+1}T^{(i+1)}S^{(i+1)}\Mod{t^{M_{i+1}}}$$ and since $W_{i+1,0}\neq 0$ it follows that $$AW_{i+2}\cdots W_s\equiv T^{(i+1)}S^{(i+1)}\Mod{t^{M_{i+1}}}$$ Inductively, we see that $A\equiv T^{(s)}S^{(s)}\Mod{t^{M_s}}$. Moreover, this is a non-trivial factorization, because $v(T_0^{(s)})\leq v(T_0)<m$, and $v(T_0^{(s)})+v(S_0^{(s)})=v(A_0)=m$, so $v(S_0^{(s)})>0$ and thus $S^{(s)}$ is not a unit. By the same argument $T^{(s)}$ is also not a unit. But since $M_s>d$, this contradicts our assumption on $d$.\\

\noindent So, fixing $M_\ell=d$, we see that 
\begin{align*}
d=M_\ell&\geq \frac{M_{\ell-1}-d_{\ell}L}{2}\geq \frac{M_{\ell-2}-d_{\ell-1}L-2d_{\ell}L}{2^2}\geq\dots\\&\geq\frac{M_0-d_1L-2d_2L-\dots-2^{\ell-1}d_\ell L}{2^\ell}
\end{align*} so it follows that 
\begin{align*}
N&\leq M_0\leq 2^\ell d+d_1L+2d_2L+\dots+2^{\ell-1}d_\ell L\leq 2^\ell d+(1+2+\dots+2^{\ell-1})d'L\\&=2^\ell d+(2^\ell-1)d'L\leq 2^md+(2^m-1)d'L\qedhere
\end{align*}

\end{proof}

\noindent In light of this lemma, for each $m>1$, and each $L\geq 1$, define the strictly increasing sequence of integers $a_0^{(m)}(L)<\dots<a_m^{(m)}(L)$ by $a_0^{(m)}(L):=(2^m-1)L+2^m>L$ and for $k>0$, $$a_k^{(m)}(L):=(2^m-1)(a_{k-1}^{(m)}(L)+1)L+2^m$$ %Then we define $N_m:=c_m=(2^m-1)(c_{m-1}+1)N_{m-1}+2^m$, and clearly $N_m>N_{m-1}$.

\noindent Our goal now is to show that if $L$ controls primality to degree $m$ then $a_m^{(m)}(L)$ controls primality to degree $m+1$. Let us first introduce the following notation:

\begin{definition}
For any non-unit $W\in V[[t]]$ with $W_0 \neq 0$ and $d\in\mathbb{N}$, let $k_d(W)$ be the maximum integer $k$ such that $W\equiv W_1\cdots W_{k+1}\Mod{t^d}$, where each $W_i$ is irreducible modulo $t^d$. Note that $k_d(W)=0$ if and only if $W$ is irreducible modulo $t^d$.
\end{definition}

\begin{lemma}\label{reduced} $ $ Let $W$ be a non-unit in $V[[t]]$, with $W_0 \neq 0$. Let $d \in \N$. Then:
\begin{enumerate}
\item $k_d(W)$ is well-defined, and $k_d(W)< v(W_0)$

\item If $W,U,Z\in V[[t]]$ are non-units, $d\geq 1$ and $W\equiv UZ\Mod{t^d}$, then $k_d(U)+k_d(Z)<k_d(W)$.
\end{enumerate}
\end{lemma}

\begin{proof} If $v(W_0)=k$ and $W\equiv W_1\cdots W_{k+1}\Mod{t^d}$ is a factorization into non-units, then $v(W_{i,0})\geq 1$ for all $i$, which is impossible. Therefore, there is a maximum possible length of factorization for $W$, and every term in such a maximum length factorization is irreducible.

Next, by definition, $U\equiv U_1\cdots U_{k_d(U)+1}$ and $Z\equiv Z_1\cdots Z_{k_d(Z)+1}\Mod{t^d}$, where $U_i,Z_i$ are irreducible modulo $t^d$. Thus $$W\equiv U_1\cdots U_{k_d(U)+1}Z_1\cdots Z_{k_d(Z)+1}\Mod{t^d}$$ So by definition, $k_d(W)\geq k_d(U)+k_d(Z)+1>k_d(U)+k_d(Z)$.\qedhere

\end{proof}

\begin{proposition}\label{N_m}
If $A\in V[[t]]$ is irreducible of height $d$, $v(A_0)=m>1$, and $L\geq 1$ controls primality to degree $m$. Then if $AW\equiv ST\Mod{t^N}$ with $v(W_0)<m$, $0<v(S_0),v(T_0)<m$, then $N<a_m^{(m)}(L)d$.
\end{proposition}

\begin{proof}
Since $m$ and $L$ are fixed, we will denote $a_k:=a_k^{(m)}(L)$ throughout for simplicity.
Since $a_m^{(m)}(L)>L\geq 1$, the statement is obvious if $N\leq d$, so we assume $N> d$ throughout. Moreover, if $v(W_0)=0$ then $W$ is a unit in $V[[t]]$. So since $S,T$ are non-units, the factorization $A\equiv (W^{-1}S)T\Mod{t^N}$ implies that $N\leq d$. So we can also assume that $v(W_0)>0$.\\

%\noindent First note that for any $M>0$, if $W\equiv W_1\cdots W_n\Mod{t^M}$, where each $W_i$ is a non-unit (i.e. $v(W_{i,0})>0$), then this yields a factorization $W_0=W_{1,0}\cdots W_{n,0}$, with $v(W_{1,0})+\dots+v(W_{n,0})=v(W_0)<m$. This implies that $n<m$. In particular, since the number of possible terms in the factorization is bounded, $W$ must decompose into a product of irreducibles modulo $t^M$, and it can require no more than $v(W_0)<m$ factorizations to obtain this decomposition.\\

\noindent Given \emph{any} factorization $W\equiv W_1\cdots W_n\Mod{t^M}$ into non-units, for $d\leq M\leq N$, let $k:=k_d(W_1)+\dots+k_d(W_n)$. Note that this yields a non-trivial factorization $W_0=W_{1,0}\cdots W_{n,0}$, with $v(W_{1,0})+\dots+v(W_{n,0})=v(W_0)<m$. This implies that $n<m$, and hence $k<m$.

We will prove by induction on $k$ that $M\leq a_k d$. Since $k\leq k_d(W)\leq v(W_0)<m$ by Lemma \ref{reduced} applied iteratively to $W\equiv W_1\cdots W_n\Mod{t^d}$, this will imply that $M<a_md$.\\

\noindent First, $k=0$ if and only if each $W_i$ is irreducible modulo $t^d$, so it follows from Lemma \ref{extra} that $M\leq 2^md+(2^m-1)Ld=a_0d$, so we will assume that $k>0$. \\

%In particular, this means that $W_i$ is not irreducible modulo $t^d$ for some $i$. But if $m=2$ then $v(W_0)=1$ so $W$ is irreducible modulo $t$, thus we may assume that $m>2$, which implies that $N_{m-1}>N_1=0$.\\

\noindent Let us first suppose that $W_i$ is not irreducible modulo $t^M$ for some $i$, and without loss of generality we may assume that $i=1$. Writing $W_1\equiv UZ\Mod{t^{M}}$, we see that $k_d(W_1)>k_d(U)+k_d(Z)$ by Lemma \ref{reduced}, so $k_d(U)+k_d(Z)+k_d(W_2)+\cdots+k_d(W_n)<k$, and thus $M\leq a_{k-1}d< a_k d$ by induction.

Therefore, we may assume that each $W_i$ is irreducible modulo $t^M$. Fix $d'\leq M$ minimal such that $W_1,\dots,W_n$ are irreducible modulo $t^{d'}$, and we may assume without loss of generality that $W_1$ is irreducible of height $d'$. We know by Lemma \ref{extra} that $M\leq 2^md+(2^m-1)Ld'$.\\

\noindent Thus $d'$ is larger than $M':= \lceil \frac{M-2^md}{(2^m-1)L}\rceil-1$, and hence $W_1$ is reducible modulo $t^{M'}$. If $M'<d$ then $M\leq 2^md+(2^m-1)Ld=a_0d\leq a_kd$, so we may assume that $M'\geq d$.\\

\noindent Again, writing $W_1\equiv UZ\Mod{t^{M'}}$, by the same argument we have $k_d(U)+k_d(Z)+k_d(W_2)+\cdots+k_d(W_n)<k$, so $M'\leq a_{k-1}d$ by induction.\\ 

\noindent So $\frac{M-2^md}{(2^m-1)L}\leq M'+1\leq a_{k-1}d+1\leq (a_{k-1}+1)d$, and thus $$M\leq 2^md+(2^m-1)(a_{k-1}+1)Ld=a_kd$$\qedhere
\end{proof}

\noindent Now, let $N_1:=1$, and for each $m>1$, define $N_m:=a_m^{(m)}(N_{m-1})$. This gives a strictly increasing chain of integers $1=N_1<N_2<\cdots$.

\begin{proposition}\label{irred implies prime (finite height)}
For each $m\geq 1$, $N_m$ controls primality to degree $m+1$.
    
\end{proposition}
\begin{proof}
We prove the statement using induction on $m$. If $m=1$, this follows from Corollary \ref{valuation_one_case}, so we will assume $m>1$ and $N_{m-1}$ controls primality to degree $m$.\\

\noindent Suppose $A\in V[[t]]$ is irreducible of height $d$ and $v(A_0)\leq m$. If $v(A_0)<m$ then $A$ is $N_{m-1}d$-prime by the inductive hypothesis, and hence it is $N_md$-prime, so we may assume that $v(A_0)=m$.

If $A|_nB_1B_2$, and choose $i,j\geq 0$ maximal such that $i,j\leq n$, $A|_iB_1$ and $A|_jB_2$. If $i+j\geq n$ then $i+j>n-N_md$ as required, so we will assume that $i+j<n$. Therefore, using Lemma \ref{auxillary}, we can find $W,S,T\in V[[t]]$ such that $v(S_0),v(T_0),v(W_0)<m$ and $AW\equiv ST\Mod{t^{n-i-j}}$.

Applying Proposition \ref{N_m}, we see that $n-i-j< a_m^{(m)}(N_{m-1})d=N_md$. Thus $i+j>n-N_md$, and $A$ is $N_md$-prime.\\

\noindent So we have proved that for every $A\in V[[t]]$ irreducible of finite height with $v(A_0)<m+1$, $A$ is $N_m\text{ht}(A)$-prime, i.e. $N_m$ controls primality to degree $m+1$.\end{proof}

\begin{proposition}\label{quasi-prime}
For all $r,D\in\mathbb{N}$, there exists a constant $C_r(D)\in\mathbb{N}$ such that if $A\in V[[t]]$ is irreducible mod $t^D$, with $v(A_0)=r$, then $A$ is $C_r(D)$-prime. For $r=1$, we may take $C_r(D)=D$.
\end{proposition}

\begin{proof}

Using Proposition \ref{irred implies prime (finite height)}, we see that $A$ is $N_r\text{ht}(A)$-prime. So since $\text{ht}(A)\leq D$, let $C_r(D):=N_rD$, and it follows immediately that $A$ is $C_r(D)$-prime. Since $N_1=1$, we also have that $C_1(D)=D$.\end{proof}

\begin{proposition} \label{quasi_divides}
    Let $V$ be a DVR. Fix $r,D,s,\ell \in \N.$ There is an integer $Q = Q(r,D,s,\ell) \geq \ell$ such that: For any $A \in V[[t]]$, irreducible mod $t^D$, with $v(A_0) = r$, if $B_1,\ldots,B_s \in V[[t]]$ satisfy $A|_Q B_1 \cdot \ldots \cdot B_s$, then $A|_\ell B_i$ for some $1\leq i \leq s$.
    
\end{proposition}

\begin{proof}
For $s = 1$ take $Q = \ell$. Let $C_r(D)$ be the constant given by Proposition \ref{quasi-prime}. Let $s >1$, suppose we have proven the claim for $s-1$, and let $Q'$ be the corresponding bound. Choose any $Q\geq \ell$ large enough so that $Q-C_r(D)\geq 2Q'$. Let $A \in V[[t]]$, irreducible mod $t^D$, with $v(A_0) = r$ and let $B_1,\ldots,B_s \in V[[t]]$ satisfy $A|_Q B_1 \cdot \ldots \cdot B_s$. Let $i$ be a maximal integer such that $A|_iB_1$ (if $A|_iB_1$ for arbitrarily large $i$ then $A|B_1$ by Lemma \ref{lem_a|b}, and we are already done). Let $j$ denote a maximal integer such that $A|_jB_2\cdot \ldots \cdot B_s$, if such an integer exists. If $i \geq \ell$ we are done. Suppose that $i < \ell$. Then 
$$j > Q-C_r(D)-i \geq Q-C_r(D)-(\ell-1)\geq 2Q'-\ell+1> Q',$$
hence $A|_{Q'} B_2\ldots B_s$. If no such an integer $j$ exists, then $A| B_2\ldots B_s$ and also in particular  $A|_{Q'} B_2\ldots B_s$. Thus the claim follows by the induction hypothesis.

\end{proof}
\noindent\textbf{Remark:} Proposition \ref{quasi-prime} can be extended, in a limited way, to non-DVRs: Let $R$ be an arbitrary UFD and let $a \in R[[t]]$. Then:\begin{enumerate}
    \item If $a_0$ is a prime in $R$, then $a$ is $1$-prime.
    \item If $a_0$ is a power of a prime, and $a$ irreducible modulo $t^2$, then $a$ is $2$-prime. 
\end{enumerate} 

These facts are straightforward to prove using a variant of Lemma \ref{auxillary}; we do not include proofs as these will not be used in the sequel. However, even if one assumes that, say, $a_0 = p^2$ is a square of a prime, and that $a$ is irreducible mod $t^3$, it does not follow that $a$ is $C$-prime for any $C$. For example, for $R=\mathbb Q[X,Y]$, let $p=X^2-Y^3$, and set $a=p^2-Yt^2\in R[[t]]$. Then one verifies that $a$ is irreducible modulo $t^3$, but $a$ is not $C$-prime for any $C$. We omit the details.

%Due to this technical obstruction we are led to first prove Theorem

\section{Theorem \ref{finite_irreducibilty_theorem} over discrete valuation rings}\label{sec:dvr_case}

For this section, we again fix a DVR $V$ with valuation $v$.

\begin{lemma}\label{finite_set}
Fix $F \in V[[t]]$ with $F_0 \neq 0$, $k = v(F_0) > 1$, and integers $1 \leq r <k$, $D \geq 1$, $\ell \geq 1$. For any $n \in \N$, let ${\cal{A}}_n$ denote the family of ideals of the form $(A,t^\ell)$, where $A \in V[[t]]$ is irreducible mod $t^D$, $v(A_0) = r$, and $A|_n F$. Then there exists an integer $M$ such that  $|\bigcup_{n \geq M}{\cal{A}}_n| \leq \floor{\frac{k}{r}}$. \end{lemma}
\begin{proof}
    Let $h = \floor{\frac{k}{r}}+1$. We shall prove that for a sufficiently large $M$, the union $\bigcup_{n \geq M}{\cal{A}}_n$ does not contain $h$ pairwise distinct ideals. 
    
    Recursively define integers $E_1,\ldots,E_h \geq 0$ as follows: Let $E_1 = 0$. Having defined $E_s$ for some $1\leq s <h$, let $Q_s = Q(r,D,s,\ell) \geq 0$ be the constant given by Proposition \ref{quasi_divides} for $r,D,s,\ell$, let $C_r(D)$ be the constant given by Proposition \ref{quasi-prime}, and let 

    $$E_{s+1} = E_s+C_r(D)+Q_s.$$
Having defined $E_1,\ldots,E_h$, let $M = E_h+1$ and suppose that $\bigcup_{n \geq M} {\cal{A}}_n$ contains $h$ distinct ideals:

$$(A_1,t^\ell), \ldots,(A_h,t^\ell),$$
where $(A_i,t^\ell) \in {\cal{A}}_{n_i}$ for a suitable integer $n_i \geq M$,  $A_i \in V[[t]]$ is irreducible mod $t^D$, $v((A_i)_0) = r$, and $A_i|_{n_i} F$, $i = 1,\ldots,h$. Let $n = \min\{n_1,\ldots,n_h\}$. Then $n \geq M > E_h$ and $A_i|_n{F}$ for every $i$.   
We prove by induction on $s = 1,\ldots,h$ that 
$$A_1\cdot \ldots \cdot A_s |_{n-E_s} F.$$
For $s = 1$ this is the assertion noted above $A_1|_n F$. Suppose we have proven the claim for some $s < h$. Write $$P_s = A_1\ldots A_s.$$
By the induction hypothesis, $F \equiv_{n-E_s} P_sH$ for some $H \in V[[t]]$. Since $A_{s+1}|_{n}F$, we have $$A_{s+1}|_{n-E_s}P_sH.$$
Let $j$ denote a maximal integer such that $A_{s+1}|_jP_s$, if such an integer exists. If $j$ exists and $j \geq Q_s$, then  $$A_{s+1}|_{Q_s} A_1\ldots A_s,$$ hence by Proposition \ref{quasi_divides} we get that $A_{s+1}|_\ell A_q$ for some $q  \leq s$. If no such $j$ exists we also have $A_{s+1}|_\ell A_q$.  But since $v((A_{s+1})_0) = v((A_q)_0) = r$ it follows that $(A_{s+1},t^\ell) = (A_q,t^\ell)$ in both cases, a contradiction. 

So suppose that $j$ exists and satisfies $j < Q_s$. Let $i$ denote maximal integer such that $A_{s+1}|_i H$, if such an integer exists. Then since $$A_{s+1}|_{n-E_s} P_sH$$ we have $i+j > n - E_s- C_r(D),$ by Proposition \ref{quasi-prime}, and then 
$$i > n-E_s-C_r(D)- Q_s = n-E_{s+1},$$ 
hence $A_{s+1} |_{n-E_{s+1}} H,$ hence 

$$P_sA_{s+1}|_{n-E_{s+1}} F.$$
If no such $i$ exists, we also have $A_{s+1} |_{n-E_{s+1}} H,$ and hence $P_sA_{s+1}|_{n-E_{s+1}} F.$ This completes our proof by induction. Now for $s = h$, we have
    $$A_1\cdot \ldots \cdot A_h |_{n-E_h}F,$$ and since $n> E_h$ the constant term of the left-hand side divides the constant term of the right-hand side. Equivalently,
    $$v((A_1)_0)+\ldots+v((A_h)_0) \leq v(F_0),$$ that is, $hr \leq k$, in contradiction with our choice of $h$.
\end{proof}
\begin{proposition}
    \label{theorem_over_dvr} Let $F \in V[[t]]$ be irreducible. Then $F$ is irreducible of finite height.
\end{proposition}

\begin{proof}
    By Lemma \ref{zero_coefficient} we may assume that $F_0 \neq 0$. Let $k = v(F_0)$. Then $k > 0$ and by multiplying with a unit we have $F_0 = p^k$, where $p$ is a uniformizer for $V$. If $k = 1$, then $F$ is irreducible modulo $t$, so assume $k > 1$.
    Suppose that $F$ is reducible modulo arbitrarily high powers of $t$. Let $1 \leq r \leq k-1$ be a minimal integer such that $F$ has a non-unit factor $A$ modulo $t^n$ with $v(A_0) = r$ for arbitrarily large $n$.     
    By multiplying with units we may assume that for arbitrarily large $n$ there exists a factorization $F \equiv AB \Mod{t^n}$ with $A_0 = p^r, B_0 = p^{k-r}$. Moreover, by the minimality of $r$, there exists an integer $D \geq 1$ such that for every large enough $n$, such a factorization $F \equiv AB \Mod{t^n}$ is with $A$ irreducible modulo $t^D$. Indeed, if no such $D$ existed, then for arbitrarily large pairs of integers $n,D'$ we could find a factorization $$F \equiv_n  AB$$ with $A_0 = p^r$ and a non-trivial factorization $$A \equiv_{ D'} CE,$$
    with $0<v(C_0) < r$; Then $$F \equiv_{\min(n, D')} CEB,$$
and by taking $n,D'$ to be both sufficiently large we can obtain arbitrarily high factors $C$ for $F$ with $0<v(C_0) < r$, a contradiction. Thus there exists $n_0$ such that every factorization $F \equiv_n AB$ with $v(A_0) = r$ and $n \geq n_0$ is with $A$ irreducible modulo $t^D$.

We now verify the conditions of Lemma \ref{tree}, with $R,f,d,e$ being $V,F,p^r,p^{k-r}$ here, respectively. The first condition of the Lemma holds by our assumption above of the existence of a factor $A$ of $F$ mod $t^n$ with $v(A_0) = r$ for arbitrarily large $n$. Indeed, if $AB \equiv F \Mod{t^n}$, then by multiplying $A,B$ with suitable units we may assume that $A_0 = p^r$, $B_0 = p^{k-r}$.

To verify the second condition, let $\ell \in \N$ and let ${\cal{A}}_{n}$ be the set of all ideals $(A,t^\ell)$, where $A \in V[[t]]$ is irreducible mod $t^D$, $A_0 = p^r$, and $A|_nF$. By Lemma \ref{finite_set} there exists $M \in \N$ such that $\bigcup_{n\geq M}{\cal{A}}_{n}$ is finite. Now let $M' = \max(M,n_0)$. Then, if $AB \equiv_n F$ with $n \geq M'$, and $A_0 = p^r$, then $(A,t^\ell) \in {\cal{A}}_n$. Hence, with the notation of Lemma \ref{tree}, we have

$$\bigcup_{n \geq M'} I_{n,\ell} \subseteq \bigcup_{n \geq M} {\cal{A}}_n,$$
and the right-hand side is finite, hence so is the left. This establishes the conditions of the lemma. Thus $F$ is reducible, a contradiction.
   
\end{proof}

\begin{lemma}
Let $F_1,\ldots,F_s\in V[[t]]$  
be irreducible elements with non-zero constant terms. For each subset
$J\subseteq \{1,\ldots,s\}$, write
$$
F_J=\prod_{j\in J}F_j,
$$
with the empty product equal to $1$. Then, for every
$i\in\{0,\ldots,s\}$ and every $\ell \geq 1$, there exists an integer
$H_i(\ell)$ with the following property: if
$$
F_1\cdots F_i\equiv_n AB, n\ge H_i(\ell),
$$
then there is a subset $J\subseteq \{1,\ldots,i\}$ such that
$$
(A,t^\ell)=(F_J,t^\ell).
$$
\end{lemma}

\begin{proof}
We prove the statement by induction on $i$.

For $i=0$, take $H_0(\ell)=\ell$. Then $F_1\cdots F_i=1$, so 
$$
1\equiv_n AB, n\ge \ell,
$$
implies that $A$ is a unit modulo $t^\ell$ and hence
$(A,t^\ell)=V[[t]]$.

Assume now that $i\ge 1$, and assume that the result has been proven for $i-1$. Put
$P=F_i,\ G=F_1\cdots F_{i-1}$ and let 
$
d_i=v(P_0)>0.
$
By Proposition \ref{theorem_over_dvr} there is $D_i\ge 1$ such that $P$ is irreducible mod $t^{D_i}$. Put $C_i=C_{d_i}(D_i)$
where $C_{d_i}(D_i)$ is the constant given by Proposition \ref{quasi-prime}. 

Choose $H_i(\ell)$ large enough so that for every $n\ge H_i(\ell)$, if
$$
n'=\left\lfloor \frac{n-C_i}{2}\right\rfloor,
$$
then
$$
n'\geq H_{i-1}(\ell),
n'\geq \ell.
$$

Suppose that $PG\equiv_n AB, n\ge H_i(\ell)$. 
Then $P|_n AB$. If $P|_r A$ for arbitrarily large $r$, then $P|A$ by Lemma \ref{lem_a|b}, and in particular $P|_{n'} A$. Thus we may assume that there exists a maximal integer $\alpha$ such that $P |_\alpha A$ . Similarly we may assume that $\beta$ is a maximal integer such that $P|_\beta B$. Then $$
\alpha+\beta>n-C_i.
$$
Therefore at least one of $\alpha,\beta$ is at least $n'$. Hence
$P |_{n'} A$ or $P|_{n'} B.$ First suppose that $P|_{n'}A$. Choose $A'\in V[[t]]$ such that $A \equiv_{n'} PA'$. Reducing $PG\equiv_n AB$ modulo $t^{n'}$, we get
$$
PG\equiv_{n'} PA'B.
$$
Since $P_0\neq 0$, we deduce that
$$
G\equiv_{n'} A'B.
$$
By the induction hypothesis applied to $G=F_1\cdots F_{i-1}$, there is a subset $J\subseteq \{1,\ldots,i-1\}$ such that
$$
(A',t^\ell)=(F_J,t^\ell).
$$

Since $n'\ge \ell$ and $A\equiv_{n'}PA'$, we have
$$
(A,t^\ell) =(PA',t^\ell) =
(PF_J,t^\ell) = (F_{J\cup\{i\}},t^\ell).
$$

Now suppose that $P|_{n'}B$. Choose $B'\in V[[t]]$ such that
$B\equiv_{n'}PB'$. Reducing $PG\equiv_n AB$ modulo $t^{n'}$, we get
$$
PG\equiv_{n'} APB'.
$$
and hence
$$
G\equiv_{n'} AB'.
$$
By the induction hypothesis applied to $G=F_1\cdots F_{i-1}$, there is a subset $J\subseteq \{1,\ldots,i-1\}$ such that
$$
(A,t^\ell)=(F_J,t^\ell).
$$

This proves the induction step and concludes the proof.
\end{proof}

\begin{lemma}\label{local_bound} Fix $F \in V[[t]]$ with $F_0 \neq 0$, $k = v(F_0) > 1$, and fix  integers $0 \leq r  \leq k$, $\ell \geq 1$. For any $n \in \N$ let 

$${\cal{B}}_n = \{(A,t^\ell) \st A\in V[[t]], v(A_0) = r, A|_nF\}.$$

There exists $M \in \N$ such that $\bigcup_{n \geq M}{\cal{B}}_n$ is finite. \end{lemma}
\begin{proof}
    %If $r = 0$ then $F$ is a unit, and then $(A,t^\ell) = V[[t]]$ for all $A$. If $r = k$ then $(A,t^\ell) = (F,t^\ell)$ whenever $A|_n F$ and $n \geq \ell$. So suppose that $0 < r< k$. 
    Since $V[[t]]$ is atomic, let us write $F = F_1\cdot \ldots \cdot F_s$, where $F_1,\ldots,F_s \in V[[t]]$ are irreducible. Let $M = H_s(\ell)$ be given by the preceding lemma. Let $n \geq M$ and let $(A,t^\ell) \in {\cal{B}}_n$. By definition, there exists $B \in V[[t]]$ such that $F \equiv_n AB$. By the assertion of the preceding lemma, we have a subset $I \subseteq \{1,\ldots,s\}$ such that $$(A,t^\ell) = (F_I,t^\ell).$$ Thus every ideal in $\bigcup_{n \geq M}{\cal{B}}_n$ belongs to the finite set $$\{(F_J,t^\ell) \st J \subseteq \{1,\ldots,s\}\},$$  hence $\bigcup_{n \geq M}{\cal{B}}_n$ is finite. 
    
\end{proof}

\section{Proof of Theorem \ref{finite_irreducibilty_theorem}}\label{sec:krull_proof}

In this section we prove Theorem \ref{finite_irreducibilty_theorem}. Let us first recall the definition of a Krull domain: A domain $R$ is called a {\it Krull domain} if:
\begin{enumerate}
    \item $\bigcap_{\frp} R_\frp = R$, where the intersection is taken over all height-1 prime ideals of $R$. (Here, $R_\frp$ denotes the localization of $R$ at $\frp$).
    \item Each $R_\frp$ is a rank-1 discrete valuation ring. We shall denote the corresponding discrete valuation by $v_\frp$. 
    \item For each $0 \neq a \in R$, there exist finitely many height-1 prime ideals $\frp$ that contain $a$.
\end{enumerate}

Recall \cite[Chapter VII, \S1]{BourbakiCA} that every Noetherian integrally closed domain is a Krull domain, that every unique factorization domain is a Krull domain, and that every Krull domain is integrally closed. Krull domains need not be Noetherian; for instance, a polynomial ring in infinitely many variables over a field is a non-Noetherian UFD, hence a Krull domain.

\begin{lemma}
    Let $R$ be a Krull domain, and let $0 \neq a \in R$. Then $a$ has finitely many divisors in $R$, up to associates. 
\end{lemma}
\begin{proof}
Let $\mathcal P$ be the finite set of height-1 prime ideals of $R$ that contain $a$. For each $\frp \in \mathcal P$, let $v_\frp$ be the corresponding discrete valuation. If $b$ divides $a$, then $0 \leq v_\frp(b) \leq v_\frp(a)$ for all $\frp \in \mathcal P$, and $v_{\mathfrak q}(b)=0$ for every height-1 prime $\mathfrak q \notin \mathcal P$. Thus there are only finitely many possible valuation vectors for $b$. Since a principal fractional ideal in a Krull domain is determined by its height-1 valuations \cite[Chapter VII, \S1]{BourbakiCA}, there are only finitely many possible principal ideals $bR$. Hence $a$ has only finitely many divisors up to associates.
\end{proof}

\begin{lemma}\label{V_to_R}
    Let $R$ be a Krull domain. Let $a,c \in R[[t]]$ with $0 \neq a_0 = c_0$. Let $\ell \geq 1$. Suppose that for every height-1 prime ideal $\frp$ of $R$ with $a_0 \in \frp$ we have $a \cong_\ell c$ in $R_\frp[[t]]$. Then $a \cong_\ell c$ in $R[[t]]$. 
\end{lemma}
\begin{proof}
    %{\color{blue} Let $K$ denote the fraction field of $R$.}
         
  Let $E$ denote the quotient field of $R$. 
Since $a_0 \neq 0$, we have $u = a^{-1}c \in E[[t]]^\times$, and $u_0 = 1$ since $a_0 = c_0$. Let $\frp$ be a height-1 prime ideal of $R$ with $a_0\in \frp$. By assumption, there exists $w \in R_\frp[[t]]^\times$ such that $c \equiv_\ell wa$, hence $ua \equiv_\ell wa$ in $E[[t]]$, hence $w \equiv_\ell u$ in $E[[t]]$, hence the first $\ell$ coefficients of $u$ belong to $R_\frp$. 

Next, if $\frp$ is a height-1 prime ideal in $R$ which does not contain $a_0$, then $a$ is a unit in $R_\frp[[t]]$, hence $u = a^{-1}c \in R_\frp[[t]]$, and in particular the first $\ell$ coefficients of $u$ belong to $R_\frp$. Thus these first $\ell$ coefficients belong to the intersection $\bigcap_\frp R_\frp$, where $\frp$ runs over all height-1 primes of $R$, hence these coefficients belong to $R$, since $R$ is a Krull domain. Since $u_0 = 1$, it follows that $c \cong_\ell a$ in $R[[t]]$.

\end{proof}

\begin{proposition}\label{finite_union}
    Let $R$ be a Krull domain, let $f \in R[[t]]$ with $f_0 \neq 0$. Let $d, e$ be non-units in $R$ with $f_0 = de$, and fix $\ell \in \N$. For any $n \in \N$, let $I_n$ be the set of ideals of the form $(a,t^\ell)$ in $R[[t]]$, where $a \in R[[t]]$ is a factor of $f$ mod $t^n$ with $a_0 = d$. Then there exists $M \in \N$ such that 
    $$\bigcup_{n \geq M} I_n$$ is finite. 
\end{proposition}
\begin{proof}

Let $\cal{P}$ be the family of height-1 prime ideals of $R$ that contain $f_0$. Since $R$ is a Krull domain, ${\cal{P}}$ is finite. For any $\frp \in {\cal{P}}$ let $V_\frp = R_{\frp}$, let $v_\frp$ denote the corresponding valuation, and let $k_\frp = v_\frp(f_0), r_\frp = v_\frp(d)\leq k_\frp $. For any $n \geq 1$ let 
$I_{n,\frp}$ be the set of ideals in $V_\frp[[t]]$ of the form $(a,t^\ell)V_\frp[[t]]$, where $a$ is a factor of $f$ mod $t^n$ in $R[[t]]$ with $a_0 = d.$ 

For each $\frp \in {\cal{P}}$ we construct a finite set $J_\frp$ of ideals in $V_\frp[[t]]$ and an integer $M_\frp \geq 1$ such that 

$$\bigcup_{n \geq M_\frp} I_{n,\frp} \subseteq J_\frp.$$
To do this, we distinguish between three cases:

If $r_\frp = 0$, then $d \in V_\frp^\times$, hence if $a\in R[[t]]$ with $a_0 = d$ then $a \in V_\frp[[t]]^\times$. In this case, we may take $M_\frp = 1$ and $J_\frp = \{V_\frp[[t]]\}$. 

Next, if $r_\frp = k_\frp$, then $e \in V_\frp^\times$. Then for any $n \geq \ell$, if $f \equiv_n ab$ in $R[[t]]$ with $a_0 = d$ then $b_0 = e \in V_\frp^\times$, and then $(a,t^\ell)V_\frp[[t]] = (f,t^\ell)V_\frp[[t]]$. Thus we may take $M_\frp = \ell$ and $J_\frp = \{(f,t^\ell)V_\frp[[t]]\}.$

Now suppose that $0 < r_\frp < k_\frp$. By Proposition \ref{local_bound} applied to $f$ as an element of $V_\frp[[t]]$, there exists $M_\frp \in \N$ such that 

$$J_\frp := \bigcup_{n \geq M_\frp}\{(A,t^\ell)V_\frp[[t]] | A \in V_\frp[[t]], v_\frp(A_0) = r_\frp, A|_n f\}$$ is finite. If $a,b \in R[[t]]$ satisfy $f \equiv_n ab$ and $a_0 = d$, then $a|_n f$ also in $V_\frp[[t]]$ and $v_\frp(a_0) = r_\frp$, hence $\bigcup_{n \geq M_\frp} I_{n,\frp} \subseteq J_\frp.$

Now let $M = \max_{\frp \in {\cal{P}}} M_\frp$. If $n \geq M$ and $(a,t^\ell) \in I_{n}$, then $$(a,t^\ell)V_\frp[[t]] \in J_\frp$$
for every $\frp \in {\cal{P}}$. Now consider the map $$\Phi \colon \bigcup_{n \geq M} I_n \to \Pi_{\frp \in {\cal{P}}} J_\frp$$ given by $$\Phi(a,t^\ell) = ((a,t^\ell)V_\frp[[t]])_{\frp \in {\cal{P}}}.$$
This map is well-defined, since if $(a,t^\ell) = (c,t^\ell)$ in $R[[t]]$ then also $(a,t^\ell)V_\frp[[t]] = (c,t^\ell)V_\frp[[t]]$. We claim that $\Phi$ is injective. Indeed, suppose that $(a,t^\ell),(c,t^\ell)$ satisfy $(a,t^\ell)V_\frp[[t]] = (c,t^\ell)V_\frp[[t]]$ for all $\frp \in {\cal{P}}$. By the definition of $I_n$, we may assume that $a_0= c_0 = d.$ Now consider any height-1 prime ideal $\frqqq$ containing $d$. Since $f_0 = de$, we have $f_0 \in \frqqq$ and hence $\frqqq \in {\cal{P}}$. Thus $(a,t^\ell) = (c,t^\ell)$ in $R[[t]]$, by Lemma \ref{V_to_R}.  Thus $\Phi$ is injective. It follows that $\bigcup_{n \geq M} I_n$ is finite, since $\Pi_{\frp \in {\cal{P}}} J_\frp$ is finite. 
\end{proof}

We are now ready to prove Theorem \ref{finite_irreducibilty_theorem}:

\begin{theorem}\label{finite_irreducibility_section}
    Let $R$ be a Krull domain. Let $f \in R[[t]]$ be irreducible. Then $f$ is irreducible of finite height. 
\end{theorem}
\begin{proof}
Without loss of generality we may assume that $f_0 \neq 0$, by Lemma \ref{zero_coefficient}. Suppose that $f$ is not irreducible of finite height. That is, suppose that modulo every $n \geq 1$, there exists $a^{(n)},b^{(n)} \in R[[t]]$ with non-unit constant terms, such that $f \equiv_n a^{(n)}b^{(n)}$. In each such factorization, we have $a^{(n)}_0b^{(n)}_0 = f_0$. Since $f_0$ has finitely many factorizations up to units, we may replace $(a^{(n)},b^{(n)})_{n \geq 1}$ with a suitable subsequence and multiply with suitable units in $R[[t]]$, to assume that $d = a^{(n)}_0$ and $e = b^{(n)}_0$ are independent of $n$. Let $\ell \in \N$. For any $n \in \N$, let $I_n$ be the set of ideals of the form $(a,t^\ell)$, where $a \in R[[t]]$ is a factor of $f$ mod $t^n$ with $a_0 = d$. By Lemma \ref{finite_union}, there exists $M \in \N$ such that 
    $$\bigcup_{n \geq M} I_n$$ is finite. Then the conditions of Lemma \ref{tree} hold, and hence there exists a factorization $f = ab$ in $R[[t]]$, where $a_0 = d,b_0 = e$ are non-units, a contradiction. 

\end{proof}

We remark that Landweber proves \cite[\S3, (3.4)]{Landweber74} that for $S = R[x_1,x_2,\ldots]$, if $f \in S[[t]]$ is irreducible, then for sufficiently large $n$, a certain projection $\varphi_n(f)$ is irreducible in a graded subring $S_n[[t]]$. In the paper's concluding remark \cite[Remark 3.9]{Landweber74}, Landweber notes that one measure of the difficulty of the problem for the full ring lies in obtaining such an irreducibility result for $S[[t]]$ . This is precisely what Theorem \ref{finite_irreducibilty_theorem} here achieves, in the general context of Krull domains.

Theorem \ref{finite_irreducibility_section} has the following immediate implication:

\begin{corollary}\label{cor:finite_height_ufd_criterion}
    Let $R$ be a UFD. Then $R[[t]]$ is a UFD if and only if every irreducible power series of finite height in $R[[t]]$ is prime. 
\end{corollary} 
Theorem \ref{finite_irreducibility_section} applies, in particular, for all Noetherian integrally closed domains, as these are all Krull domains. 

Let us also note that the theorem does not hold over arbitrary domains. For example, take a field $K$, consider the polynomial ring $$R = K[x_0,x_1,x_2,\ldots,y_0,y_1,\ldots]$$ where the $x_i$ and $y_j$ are countably many independent variables. Let $I$ be the ideal generated by $x_ny_n-x_{n-1}$ for all $n \geq 1$. One checks that $I$ is a prime ideal in $R$, and hence the quotient $S = R/I$ is a domain. Now consider the power series

$$f = x_0+x_1t+x_2t^2+\ldots \in S[[t]].$$
This element has a non-trivial factorization modulo every $t^n$:

$$f \equiv x_n\cdot\big(y_1y_2 \cdots  y_n+y_2 \cdots y_n t + \ldots +y_nt^{n-1}\big) \Mod{t^n},$$ however one verifies that $f$ itself is irreducible in $S[[t]]$. We omit the details, as this will not be needed in the sequel.

Finally, let us remark that the conclusion of Theorem \ref{finite_irreducibility_section} does not characterize Krull domains. For example, consider the polynomial ring $\F_q[x]$ over a finite field, and its subring $R= \F_q[x^2,x^{3}]$. Then every irreducible element in $R[[t]]$ is irreducible of finite height, by Proposition \ref{finite_quotient}. However, $R$ is not a Krull domain, since $R$ is not integrally closed. It would be of independent interest to characterize the domains $R$ for which the assertion of Theorem \ref{finite_irreducibility_section} holds. However, we doubt that there exists such a characterization purely in terms of the arithmetic of the ring $R$ itself. It is also unknown to the authors whether the assertion of Theorem \ref{finite_irreducibility_section} theorem holds for all Noetherian domains.

\section{A finite-stage criterion and Landweber's problem over regular UFDs}\label{sec:finite_means_prime_proof}

We now prove the finite-stage criterion stated in the introduction, and then deduce the polynomial-ring theorem over regular UFDs.

\begin{lemma}\label{retraction_units}
Let $S\subseteq R$ be domains and let $\rho:R\to S$ be a retraction, that is, $\rho|_S=\id_S$. If $s\in S$ is a unit in $R$, then $s$ is a unit in $S$. Consequently, a non-unit of $S[[t]]$ remains a non-unit in $R[[t]]$.
\end{lemma}

\begin{proof}
If $sr=1$ for some $r\in R$, then applying $\rho$ yields $s\rho(r)=1$, hence $s$ is a unit in $S$. The assertion for power series follows coefficient-wise, since a power series is a unit precisely when its constant term is a unit.
\end{proof}

\begin{theorem}\label{finite_stage_retract_criterion}
Let $R$ be a Krull domain. Suppose that for every finite subset $E\subset R$ there exist a subring $S\subseteq R$ and a retraction $\rho:R\to S$ such that $E\subseteq S$ and $S[[t]]$ is a UFD. Then $R[[t]]$ is a UFD.
\end{theorem}

\begin{proof}
Since $R$ is a Krull domain, $R[[t]]$ is a Krull domain, and hence is atomic. It remains to show that every irreducible element of $R[[t]]$ is prime.

Let $f\in R[[t]]$ be irreducible. If $f_0=0$, then by Lemma \ref{zero_coefficient} and the equality $f=t(t^{-1}f)$, the element $t^{-1}f$ is a unit; hence $f$ is associate to $t$, and is prime. Thus we may assume $f_0\neq 0$.

By Theorem \ref{finite_irreducibility_section}, the element $f$ is irreducible of finite height. Choose $m_f\ge 1$ such that $f$ is irreducible modulo $t^{m_f}$. Suppose, toward a contradiction, that $f$ is not prime. Then there exist $g,h,q\in R[[t]]$ such that
$$
fq=gh,
$$
but $f\nmid g$ and $f\nmid h$. By Lemma \ref{lem_a|b}, there exist integers $m_g,m_h\ge 1$ such that
$$
g\notin (f,t^{m_g})\quad\text{and}\quad h\notin (f,t^{m_h}).
$$
Put $m=\max\{m_f,m_g,m_h\}+1$.

Let $E\subset R$ be the finite set consisting of the coefficients of $f,g,h,q$ of degrees $<m$. By assumption, there are a subring $S\subseteq R$ and a retraction $\rho:R\to S$ with $E\subseteq S$ and $S[[t]]$ a UFD. Extend $\rho$ coefficientwise to a retraction $R[[t]]\to S[[t]]$. Set
$$
F=\rho(f),\qquad G=\rho(g),\qquad H=\rho(h),\qquad Q=\rho(q).
$$
Since $\rho$ fixes the coefficients in $E$, we have
$$
F\equiv f,
\quad G\equiv g,
\quad H\equiv h,
\quad Q\equiv q \Mod{t^m}.
$$
Applying $\rho$ to $fq=gh$ gives
$$
FQ=GH
$$
in $S[[t]]$.

We claim that $F$ is irreducible in $S[[t]]$. Indeed, suppose $F=AB$ with $A,B\in S[[t]]$ non-units. By Lemma \ref{retraction_units}, $A$ and $B$ are also non-units in $R[[t]]$. Reducing modulo $t^m$, and using $F\equiv f\Mod{t^m}$, we obtain a non-trivial factorization of $f$ modulo $t^m$ in $R[[t]]$, contradicting the choice of $m\ge m_f$.

Thus $F$ is irreducible in the UFD $S[[t]]$, hence $F$ is prime. From $FQ=GH$, we may assume, without loss of generality, that $F$ divides $G$ in $S[[t]]$. Choose $Y\in S[[t]]$ with $G=FY$. Since $F\equiv f$ and $G\equiv g$ modulo $t^m$, we have
$$
g\equiv fY \Mod{t^m}
$$
in $R[[t]]$. Hence $g\in (f,t^m)$, contradicting the choice of $m>m_g$. The case $F\mid H$ similarly contradicts the choice of $m>m_h$. Therefore $f$ is prime, and $R[[t]]$ is a UFD.
\end{proof}

\begin{theorem}\label{regular_UFD_polynomial_theorem}
Let $A$ be a regular UFD, let $I$ be any set, and put
$$
R=A[x_i\mid i\in I].
$$
Then $R[[t]]$ is a UFD.
\end{theorem}

\begin{proof}
First note that $R$ is a UFD. Indeed, every nonzero element of $R$ belongs to a finite polynomial ring over $A$. If such an element admits a factorization in $R$, then all factors involved lie in a larger finite polynomial ring over $A$, where unique factorization holds. Thus the usual finite-variable unique factorizations are also the unique factorizations in $R$. In particular, $R$ is a Krull domain.

Let $E\subset R$ be finite. Only finitely many variables occur in the elements of $E$; let $J\subset I$ be a finite set containing all of them, and put $S=A[x_j\mid j\in J]$. Then $S$ is a regular UFD. By the theorem of Samuel and Buchsbaum \cite[Theorem 2.1]{Samuel61}, \cite[Theorem 3.2]{Buchsbaum61}, the power series ring $S[[t]]$ is a UFD.

Define $\rho:R\to S$ to be the $A$-algebra homomorphism which fixes $x_j$ for $j\in J$ and sends $x_i$ to $0$ for $i\notin J$. Then $\rho$ is a retraction and $E\subseteq S$. Theorem \ref{finite_stage_retract_criterion} applies and gives that $R[[t]]$ is a UFD.
\end{proof}

Taking $I=\mathbb N$ in theorem \ref{regular_UFD_polynomial_theorem}, we obtain the solution to Landweber's problem. Let us note that in the special case where $A = K$ is a field extension of $\Q$, one can get a solution along the lines presented in this section, building upon the special case of Theorem \ref{finite_irreducibilty_theorem} proven by Bayart in \cite[\S IV]{Bayart81}. However, the assumption that all integers are units seems critical to Bayart's proof -- the general version proven here covers the case where $A = K$ is a field of positive characteristic, and for example the cases where $A = \mathcal{O}[y_1,\ldots,y_d]$ is a polynomial ring over any Dedekind domain.

The above results present progress with the basic question raised by Krull -- when is the power series ring $R[[t]]$ a UFD? However, the following natural question is left open: If $R$ is a domain such that $S = R[[x]]$ is a UFD, is $S[[t]]$ a UFD? This question was raised by Bayart \cite{Bayart73}, and is highlighted by Gilmer in \cite{Gilmer06Questions}; it remains unresolved.

\bibliographystyle{numeric}

\end{document}